\documentclass[12pt]{amsart}
\usepackage{latexsym,amssymb,verbatim,multirow,graphicx,epsfig,enumerate, paralist,hyperref}
\usepackage{xcolor}
\usepackage[normalem]{ulem}

\theoremstyle{plain}
   \newtheorem{theorem}{Theorem}
   \newtheorem{proposition}[theorem]{Proposition}
   \newtheorem{prop}[theorem]{Proposition}

   \newtheorem*{theorem*}{Theorem}
   \newtheorem*{maintheorem}{Main Theorem}
\theoremstyle{definition}

   \newtheorem{question}[theorem]{Question}
   \newtheorem{remark}[theorem]{Remark}

\numberwithin{equation}{section}

\newcommand\Symm{\mathfrak{S}}

\newcommand\rank{\operatorname{rank}}

\newcommand\fix{\operatorname{fix}}

\newcommand{\defn}[1]{\textbf{#1}}

\newcommand\ZZ{{\mathbb{Z}}}

\newcommand\RR{{\mathbb{R}}}
\newcommand\FF{{\mathbb{F}}}

\newcommand\GL{{\mathrm{GL}}}
\newcommand{\GLnFq}{\GL_n(\Fq)}

\newcommand{\Fq}{\FF_q}
\newcommand{\op}{\operatorname}

\newcommand{\bmat}[1]{\begin{bmatrix} #1 \end{bmatrix}}

\title{$\GLnFq$-analogues of some properties of $n$-cycles in $\Symm_n$}
\author{Joel Brewster Lewis}

\begin{document}

\begin{abstract}
We give analogues in the finite general linear group of two elementary results concerning long cycles and transpositions in the symmetric group: first, that the long cycles are precisely the elements whose minimum-length factorizations into transpositions yield a generating set, and second, that a long cycle together with an appropriate transposition generates the whole symmetric group.
\end{abstract}

\maketitle

The goal of this paper is to give analogues in the finite general linear group $\GLnFq$ of the following well known theorems about the behavior of transpositions and $n$-cycles in the symmetric group $\Symm_n$.
\begin{theorem*}
\begin{enumerate}
\item A permutation $w \in \Symm_n$ is an $n$-cycle if and only if the factors in every minimum-length factorization of $w$ as a product of transpositions form a generating set for $\Symm_n$.
\item If $c$ is an $n$-cycle and $t = (a \ b)$ is a transposition that exchanges two adjacent entries of $c$, then $\langle c, t \rangle = \Symm_n$.
\end{enumerate}
\end{theorem*}

Following \cite{LW}, we call the equivalent property in the first part of the theorem the \defn{strong quasi-Coxeter property}.  The second half of the theorem can easily be generalized to the case that the cyclic distance between $a$ and $b$ in the cycle $c$ is relatively prime to $n$; it is occasionally mis-stated (e.g., in \cite[Ch.~III, ex.~(8)]{ledermann}) to assert that \emph{any} transposition will do, but it is easy to see (for example) that the transposition $(1 \ 3)$ normalizes the subgroup generated by the $4$-cycle $(1 \ 2 \ 3 \ 4)$, and consequently that these two elements generate a group of order $8$ rather than the whole symmetric group $\Symm_4$.

In formulating our $q$-analogues, we follow the perspective of \cite{CSP, ReinerStantonWebb, LRS, LewisMorales, HLR}): first, we view the group $\Symm_n$ as a real reflection group, acting on $\RR^n$ by permutation of coordinates.  From this perspective, the transpositions are precisely the \defn{reflections}, i.e., those elements that fix a \defn{hyperplane} (a subspace of codimension $1$) pointwise.\footnote{It is conventional but not necessary to specify that real reflections be orthogonal, or that they have order $2$.  For a finite subgroup of $\GL_n(\RR)$, these properties are implied (after a standard averaging argument) by the seemingly more general definition here. }  In $\GLnFq$, we extend this definition verbatim and define a reflection to be any element that fixes a hyperplane pointwise. Then $\GLnFq$ is also generated by its subset of reflections. Second, we take the analogue of an $n$-cycle to be a \defn{Singer cycle} for $\GLnFq$.  These elements can be characterized in many ways; we have that $c \in \GLnFq$ is a Singer cycle if any of the following equivalent statements hold:
\begin{itemize}
\item $c$ is \defn{irreducible} (i.e., it stabilizes no nontrivial subspace of $V := \Fq^n$) of maximum possible multiplicative order;
\item $c$ has multiplicative order $q^n - 1$;
\item $c$ acts transitively on $V \smallsetminus \{0\}$; and
\item $c$ has as eigenvalue over $\overline{\Fq}$ one of the cyclic generators for $\FF_{q^n}^\times$.
\end{itemize}
Our main theorem is as follows.

\begin{maintheorem}\label{thm:main}
\begin{enumerate}
\item An element $g \in \GLnFq$ is a Singer cycle if and only if the factors in every minimum-length factorization of $g$ as a product of reflections form a generating set for $\GLnFq$.
\item Suppose that $c \in \GLnFq$ is a Singer cycle and $t \in \GLnFq$ is a reflection that does not normalize the cyclic subgroup $\langle c\rangle$; then $\langle c, t\rangle = \GLnFq$.
\end{enumerate}
\end{maintheorem}

The structure of the paper is as follows.  We begin with some background on finite fields, Singer cycles, and the group $\GLnFq$.  We then prove part (2) of the main theorem as Theorem~\ref{conj:cycle and reflection}.  As part of the proof, we give a concrete description of the reflections $t$ that normalize a Singer cycle (Proposition~\ref{prop:normalizing reflections when n = 2}).  In particular, such reflections exist only when $n = 2$, so that for $n \geq 3$ we have that any Singer cycle and any reflection together generate $\GLnFq$.  We then use part (2) to prove part (1) of the main theorem as Theorem~\ref{thm:strong qc}.  We end with a few open questions and remarks.

\section{Background}

We begin with some basic facts about finite fields (as found in, e.g., \cite[Chs.~1.1 \& 2.1]{Hou}), Singer cycles, and irreducible elements.


Fix a prime power $q$ and let $\Fq$ denote the field of cardinality $q$.  For any $n > 1$ and any irreducible polynomial $f = x^n + a_{n - 1} x^{n - 1} + \ldots + a_0 \in \Fq[x]$, the quotient $\Fq[x] / (f)$ of the polynomial ring by the ideal generated by $f$ is the (unique up to isomorphism) finite field $\FF_{q^n}$ of order $q^n$, generated as a ring over $\Fq$ by the equivalence class of $x$.  This field is the splitting field of the irreducible polynomial $f$, and $f$ is always separable (i.e., it has $n$ distinct roots in $\FF_{q^n}$).  
The unique isomorphic copy of $\Fq$ inside $\FF_{q^n}$ may be identified as the fixed set of the \defn{Frobenius automorphism}
\begin{align*}
F \colon \quad \FF_{q^n} & \to \FF_{q^n} \\
\alpha & \mapsto \alpha^q
\end{align*}
If $f$ is an irreducible polynomial of degree $n$ over $\Fq$ and $\alpha \in \FF_{q^n}$ is a root of $f$, then the complete set of roots of $f$ is the Frobenius orbit $\{\alpha, F(\alpha) = \alpha^q, F^2(\alpha) = \alpha^{q^2}, \ldots, F^{n - 1}(\alpha) = \alpha^{q^{n - 1}}\}$ of $\alpha$.

The embedding $\Fq \hookrightarrow \FF_{q^n}$ endows the latter field with the structure of a vector space over $\Fq$.  For a fixed element $\alpha$ of $\FF_{q^n}$, multiplication by $\alpha$ is $\Fq$-linear, and, consequently, if we choose any ordered basis $B$ for $\FF_{q^n}$ over $\Fq$, we get an injection
\begin{equation}\label{eq:inclusion}
\FF_{q^n}^\times \hookrightarrow \GL_{\Fq}(\FF_{q^n}) \cong \GLnFq
\end{equation}
from the multiplicative subgroup $\FF_{q^n}^\times$ into the group $\GLnFq$ of invertible $n \times n$ matrices over $\Fq$.  The minimal polynomial of an element $\alpha \in \FF_{q^n}$ over $\Fq$ is equal to the minimal polynomial of its embedded image in $\GLnFq$; in particular, if $\alpha$ is a field generator for $\FF_{q^n}$ over $\Fq$, its image in $\GLnFq$ has irreducible characteristic polynomial.  Such elements of $\GLnFq$ may be characterized in many ways.

\begin{prop}[{\cite[Prop.~4.4]{LRS}}]
For an element $g$ of $\GLnFq$, the following are equivalent:
\begin{itemize}
\item for some choice of basis $B$ for $\FF_{q^n}$ over $\Fq$ and some generator $\alpha$ for $\FF_{q^n}$ over $\Fq$, $g$ is the image of $\alpha$ under the map \eqref{eq:inclusion};
\item the characteristic polynomial of $g$ is irreducible in $\Fq[x]$; and
\item $g$ stabilizes no nontrivial proper subspace of $\Fq^n$
.
\end{itemize}
\end{prop}

The multiplicative subgroup $\FF_{q^n}^\times$ is a cyclic group.  The cyclic generators for $\FF_{q^n}^\times$ are called \defn{primitive elements}; each primitive element is obviously a field generator for $\FF_{q^n}$ over $\Fq$.  The minimal polynomial of a primitive element over $\Fq$ (necessarily irreducible of degree $n$) is called a \defn{primitive polynomial}.  The images of primitive elements under any of the maps in Equation~\eqref{eq:inclusion} are called \defn{Singer cycles}.  These elements of $\GLnFq$ may also be characterized in many ways.

\begin{prop}[{see \cite[\S2.1]{Brookfield} and \cite[Lem.~3]{Gill}}]
For an element $g$ of $\GLnFq$, the following are equivalent:
\begin{itemize}
\item for some choice of basis $B$ for $\FF_{q^n}$ over $\Fq$ and some primitive element $\zeta$ for $\FF_{q^n}^\times$, $g$ is the image of $\zeta$ under the map \eqref{eq:inclusion};
\item $g$ is irreducible and has maximum multiplicative order among the irreducible elements;
\item $g$ has multiplicative order $q^n - 1$;
\item the characteristic polynomial of $g$ is primitive in $\Fq[x]$; 
\item $g$ acts transitively on $\Fq^n \smallsetminus \{0\}$; and
\item $g$ has as an eigenvalue over $\overline{\Fq}$ a primitive element of $\FF_{q^n}$.
\end{itemize}
\end{prop}


Our main result makes reference to the minimum number $k$ such that a given element $g \in \GLnFq$ can be written as a product $g = t_1 \cdots t_k$ of $k$ reflections.  This number (called the \defn{reflection length} of $g$) has a simple intrinsic formula.

\begin{proposition}[{\cite{Dieudonne}, \cite[Prop.~2.16]{HLR}}]\label{prop:length}
Consider $V$ a finite-dimensional vector space, and $g$ an element of $\GL(V)$.  Then the minimum length of any factorization of $g$ as a product of reflections is $\dim(V) - \dim\fix(g)$.
\end{proposition}
\begin{remark}\label{remark:length}
Since $\fix(t \cdot u) \supseteq \fix(t) \cap \fix(u)$,  in any factorization $g = t_1 \cdots t_k$ of $g$ as a product of $k := \dim(V) - \dim\fix(g)$ reflections, we have the equality $\fix(g) = \bigcap_{i = 1}^k \fix(t_i)$, and so in particular each factor $t_i$ fixes every vector fixed by $g$.
\end{remark}


Given a monic polynomial $f = x^n + \ldots + a_1x + a_0$ over some field $\FF$, its \defn{companion matrix} $C_f$ is the $n \times n$ matrix 
\[
C_f := \begin{bmatrix}
&&& -a_0 \\
1 &&& -a_1 \\
& \ddots && \vdots \\
&&1& -a_{n - 1}
\end{bmatrix}
\]
over $\FF$.  It is easy to see that $f$ is the characteristic polynomial of $C_f$.  In the case that $f$ is irreducible, it follows that $C_f$ is the rational canonical form of any matrix with characteristic polynomial $f$, and so in particular that every irreducible matrix is similar to the companion matrix of its characteristic polynomial.

We will make use of the following corrected version of \cite[Thm.~2]{Gill}.

\begin{theorem}\label{thm:gill}
Let $f, g \in \Fq[x]$ be distinct monic polynomials of degree $n$ such that $f$ is primitive and the constant term of $g$ is nonzero.  Then the group $\langle C_f, C_g \rangle$ generated by the companion matrices $C_f, C_g$ of $f$ and $g$ is equal to $\GLnFq$ unless $n = 2$ and $C_g$ belongs to the normalizer of $\langle C_f\rangle$ in $\GL_2(\Fq)$.
\end{theorem}

For a discussion of how and why Theorem~\ref{thm:gill} differs from \cite[Thm.~2]{Gill}, see Section~\ref{sec:gill} below.

\section{Proof of the main theorem}

Our strategy is to prove part (2) first, then use it to prove (1).  We begin with a concrete example in $\GL_2(\Fq)$ of a situation in which a Singer cycle and reflection fail to generate the whole group; it is perhaps analogous to the example of $(1 \ 2 \ 3 \ 4)$ and $(1 \ 3)$ in $\Symm_4$.

\begin{prop}
\label{prop:normalizing reflections when n = 2}
Let $c$ be a Singer cycle in $\GL_2(\Fq)$ and let $\zeta$ be one of its eigenvalues over $\FF_{q^2}$.
\begin{enumerate}[(a)]
\item There is a basis for $\Fq^2$ in which $c$ has matrix $\begin{bmatrix} 0 & -\zeta^{q + 1} \\ 1 & \zeta + \zeta^q \end{bmatrix}$.  In this basis, $\begin{bmatrix} 1 & 0 \\ -\zeta^{-1} - \zeta^{-q} & -1 \end{bmatrix}$ is the matrix of a reflection $t$ such that $t^2 = 1$ and $tct = c^q$.  
\item With $c, t$ as in (a), $\langle t, c\rangle$ is a proper subgroup of $\GL_2(\Fq)$ 
for any $q > 2$.
\item With $c, t$ as in (a), if $t'$ is a reflection that normalizes $\langle c \rangle$, then $t' = c^k t c^{-k}$ for some integer $k$.
\end{enumerate}
\end{prop}
\begin{proof}
Since $c$ has irreducible characteristic polynomial, the second eigenvalue of $c$ is $\zeta^q$, and so the characteristic polynomial of $c$ is $f := x^2 - (\zeta + \zeta^q)x + \zeta^{q + 1} \in \Fq[x]$.  Since $c$ is irreducible, there is a basis for $\Fq^2$ in which $c$ has as matrix the companion matrix $C_f$ of this polynomial.  Thus the requested choice of basis in (a) exists, and the linear transformation $t$ is well defined.  The given matrix for $t$ has characteristic polynomial $(t - 1)(t + 1)$, so has eigenvalues $1$ and $-1$.  When $q$ is odd, it follows immediately that $t$ is a reflection and that $t^2 = 1$.  When $q$ is a power of $2$, we still must rule out the possibility that $t$ is the identity; observe that $\zeta$ has order $q^2 - 1 > q - 1$, so $\zeta^{-1} \neq \zeta^{-q}$, and consequently the given matrix for $t$ is not the identity matrix.  By a direct computation, we have that the matrix of $tct$ is $\begin{bmatrix} \zeta + \zeta^q & \zeta^{q + 1} \\
-1 & 0\end{bmatrix}$.  To see that this is the matrix of $c^q$, it is enough to check that, extending scalars to $\FF_{q^2}$, the eigenvectors of $c$ corresponding to the eigenvalues $\zeta$ and $\zeta^q$ are respectively $\bmat{-\zeta^q \\ 1}$ and $\bmat{-\zeta \\ 1}$, while these two vectors are eigenvectors for $tct$ with respective eigenvalues $\zeta^q$ and $\zeta = (\zeta^{q})^q$.  This proves part (a).

Choose an element $g$ of $\langle t, c\rangle$.  Since $t$ and $c$ are of finite order, $g$ can be written as a product of positive integer powers of $t$ and $c$.  By repeatedly applying the relation $ct = tc^q$ from (a), we can convert this to an expression of the form $g = t^i c^j$ for some $i, j \in \ZZ_{\geq 0}$.  Since there are only $2(q^2 - 1)$ distinct expressions of this form, the group $\langle t, c\rangle$ has at most $2(q^2 - 1)$ elements.  For $q > 2$, we have $2(q^2 - 1) < (q^2 - q)(q^2 - 1) = |\GL_2(\Fq)|$, and consequently $\langle t, c \rangle \subsetneq \GL_2(\Fq)$.  This proves (b).

Since $t$ normalizes $\langle c \rangle$ (by part (a)), so too do all of its conjugates $c^k t c^{-k}$.  If $t'$ is any reflection that normalizes $\langle c \rangle$, it must either preserve the two eigenspaces of $c$ or exchange them.  The first is not possible because the fixed space of $t'$ is an eigenspace of $t'$ that is distinct from the two eigenspaces of $c$ (it is defined by an equation over $\Fq$), and a linear transformation on a two-dimensional space can only have three distinct eigenlines if it is a scalar transformation.  If $t'$ exchanges the two eigenspaces of $c$, then its matrix in the eigenbasis of $c$ is $\bmat{0 & \alpha \\ \beta  & 0}$ for some $\alpha,\beta\in\FF_{q^2}$.  It is easy to check that such a matrix can be a reflection if and only if $\beta =\alpha^{-1}$.  After a computation, we see that in this case, the fixed space of $t'$ is spanned by the vector $\bmat{-\alpha\zeta - \zeta^q \\ \alpha+1}$ (in the standard basis).  The span of this vector must be defined over $\Fq$; hence, there are exactly $q + 1$ possible values of $\alpha$ (either $\alpha = -1$ or $\alpha$ is the solution to $-\alpha \zeta - \zeta^q = a(\alpha + 1)$ for some $a\in \Fq$).  But we already have $q+1$ distinct reflections that normalize $\langle c\rangle$, namely, $c^k t c^{-k}$ for $k =0, \ldots, q$.  Thus $t'$ must be one of these, as claimed.
\end{proof}

The next result shows that there are no ``bad'' reflections other than those that appear in Proposition~\ref{prop:normalizing reflections when n = 2}.

\begin{theorem}
\label{conj:cycle and reflection}
Fix a Singer cycle $c$ in $\GLnFq$.  If $t$ is a reflection in $\GLnFq$ that is not covered by Proposition~\ref{prop:normalizing reflections when n = 2} (in particular, for any reflection $t$ if $n \geq 3$), then $\langle c, t\rangle = \GLnFq$.
\end{theorem}

\begin{proof}
Fix a prime power $q$ and an $n$-dimensional vector space $V$ over $q$.  Let $c$ be a Singer cycle in $G = \GL(V)$, and choose a basis in which the matrix of $c$ is the companion matrix of its characteristic polynomial $f = x^n + a_{n-1}x^{n-1} + \ldots + a_1x + a_0$ (which is necessarily primitive):
\[
c = \begin{bmatrix}
&&& -a_0 \\
1 &&& -a_1 \\
& \ddots && \vdots \\
&&1& -a_{n - 1}
\end{bmatrix}.
\]
Let $t$ be any reflection in $G$, with fixed plane $H$.  Since $c$ acts transitively on the hyperplanes in $V$, there is some number $k$ so that $c^k H$ is given by the coordinate equation $x_n = 0$.  In this case, the element $t' = c^k t c^{-k}$ is a reflection with fixed space $x_n = 0$, so in coordinates we have
\[
t' = \begin{bmatrix}
1 &&& b_0 \\
& \ddots && \vdots \\
&&1& b_{n - 2} \\
&&& b_{n - 1}
\end{bmatrix}
\]
for some $(b_0, \ldots, b_{n - 1}) \in \Fq$, with $b_{n - 1} \neq 0$.  By matrix multiplication, we have
\[
c \cdot t' = \begin{bmatrix}
&&& c_0 \\
1 &&& c_1 \\
& \ddots && \vdots \\
&&1& c_{n - 1}
\end{bmatrix}
\]
for some $(c_0, \ldots, c_{n - 1}) \in \Fq$, with $c_0 = -a_0 b_{n - 1} \neq 0$.  Thus $c \cdot t'$ is the companion matrix of an degree-$n$ polynomial with nonzero constant term.  And, of course, $t'$ is not the identity and so $c \cdot t' \neq c$.  Thus, by Theorem~\ref{thm:gill}, we have two possibilities: either 
\[
\GLnFq = \langle c, c \cdot t' \rangle = \langle c, c^{k + 1} t c^{-k} \rangle = \langle c, t \rangle,
\]
as desired, or $n = 2$ and $c \cdot t'$ belongs to the normalizer of $\langle c \rangle$.  In the latter case, $t = c^{-k}t'c^k$ also belongs to the normalizer of $\langle c\rangle$; these reflections are completely described by Proposition~\ref{prop:normalizing reflections when n = 2}.
\end{proof}

Having proved part (2) of the Main Theorem, we now complete the other half.

\begin{theorem}\label{thm:strong qc}
An element $g \in \GLnFq$ is a Singer cycle if and only if $g$ is strongly quasi-Coxeter, i.e., if and only if in every minimum-length reflection factorization of $g$, the set of factors generates $\GLnFq$.
\end{theorem}
\begin{proof}
First, we show that every Singer cycle is strongly quasi-Coxeter.  Let $c$ be a Singer cycle in $\GLnFq$.  If $n \geq 3$, consider any minimum-length factorization $c = t_1\cdots t_n$ into reflections.  (The length of the factorization is $n$ by Proposition~\ref{prop:length}.)  Obviously $\langle t_1, c \rangle = \langle t_1, t_1\cdots t_n \rangle \subseteq \langle t_1, \ldots, t_n\rangle \subseteq \GLnFq$.  But by Theorem~\ref{conj:cycle and reflection}, $\langle t_1, c \rangle = \GLnFq$, so the previous containment is equality.  We now consider smaller $n$.

When $n = 1$, the Singer cycle $c$ is a reflection (as is every non-identity element of $\GL_1(\Fq) = \Fq^\times$), so the unique minimum-length reflection factorization of $c$ is the trivial factorization $c = c$, and indeed $\langle c \rangle = \GL_1(\Fq)$ in this case.  When $n = 2$ and $q = 2$, we have $\GL_2(\FF_2) \cong \Symm_3$ permutes the three nonzero vectors in $\FF_2^2$, the Singer cycles are the $3$-cycles, and the reflections are the transpositions, so the statement is a particular case of the theorem mentioned in the introduction.  When $n = 2$ and $q > 2$, we proceed as in the case $n \geq 3$, but we must rule out the possibility that $t_1$ (the first factor in a minimum-length reflection factorization $c = t_1\cdot t_2$) is one of the $q + 1$ reflections that normalize $c$.  If this were the case, then $t_2 = t_1^{-1}c$ would also belong to the normalizer.  By Proposition~\ref{prop:normalizing reflections when n = 2}, it would follow that $\det(t_1)=\det(t_2) = -1$ and so that $\det(c)=1$; but this contradicts the fact that the determinant of $c$ is always a cyclic generator of $\Fq^\times$.

We now consider the converse.  Let $g$ be an element of $\GLnFq$ that is not a Singer cycle; we must show that $g$ is not strongly quasi-Coxeter.

First, suppose that $g$ is not irreducible; then there is some nontrivial subspace $W$ of $V = \Fq^n$ stabilized by $g$. We now construct a minimum-length reflection factorization of $g$ whose factors all stabilize $W$.  Let $U$ be any complementary subspace to $W$, so that $U \oplus W = V$, and consider the restriction $g|_W$ of $g$ to $W$.  Choose a minimum-length reflection factorization $g|_W = t_1 \cdots t_k$ of $g|_W$ in $\GL(W)$; by Proposition~\ref{prop:length}, its length is $k := \dim(W) - \dim\fix(g|_W)$.  Each of the reflections $t_i$ may be extended to a reflection in $\GL(V)$ by defining $t_i(u) = u$ for all $u \in U$ and extending by linearity.  Each of these (extended) reflections stabilizes $W$.  Now consider the element $g' := t_k^{-1} \cdots t_1^{-1} g $.  By construction, this element fixes $W$ pointwise and has fixed space dimension $\dim \fix(g') = \dim \fix(g) + k$.  Therefore, by the strengthening of Proposition~\ref{prop:length} in Remark~\ref{remark:length}, $g'$ can be written as a product $g' = t_{k + 1} \cdots t_{k + m}$ of $m := \dim(V) - \dim\fix(g') = \dim(V) - \dim\fix(g) - k$ reflections, each of which fixes $\fix(g')$ (so in particular $W$) pointwise.   Then $t_1 \cdots t_{k + m}$ is a factorization of $g$ into $k + m = \dim(V) - \dim\fix(g)$ reflections, each of which stabilizes $W$.  By Proposition~\ref{prop:length}, this is a minimum-length reflection factorization of $g$.  Since each factor stabilizes $W$, the same is true of every element of the group generated by these factors; and since $\GLnFq$ does \emph{not} stabilize $W$, it follows that $g$ is not strongly quasi-Coxeter.  

Second, suppose that $g$ is irreducible but its determinant $d$ is not a cyclic generator for $\Fq^\times$.  Since $g$ is irreducible, it has trivial fixed space, and so by Proposition~\ref{prop:length} it can be written as a product of $n$ reflections in $\GLnFq$ but no fewer.  Let $X \subsetneq \Fq^\times$ be the proper subgroup of generated by $d$, and let $G := \{w \in \GLnFq \colon \det(w) \in X\}$ be the proper subgroup of $\GLnFq$ consisting of all linear transformations on $V$ whose determinants belong to $X$.  As shown in \cite{LVinroot}, there exists a factorization of $g$ (in fact, precisely $|g|^{n - 1}$ factorizations) as a product of $n$ reflections \emph{in the proper subgroup $G$}.  The subgroup generated by such a factorization is a subgroup of $G$, hence a proper subgroup of $\GLnFq$; it follows that $g$ is not strongly quasi-Coxeter in this case, either.
\end{proof}

\section{Further remarks}

\subsection{Other groups}
Suppose that $G \subset \GLnFq$ is a reflection group that contains irreducible elements.\footnote{For the \emph{classical groups}, these were classified by Huppert \cite{Huppert}.  They include  the unitary groups $\op{GU}_{2n + 1}(\FF_{q^2})$ of odd dimension, the symplectic groups, and certain orthogonal groups; \cite[Table 1]{Bereczky} gives a convenient summary.  At least in principle, one can extract a full list of reflection groups over $\Fq$ that contain a Singer cycle by combining Huppert's theorem with the classification of irreducible reflection groups over $\Fq$  (as summarized, e.g., in \cite{KemperMalle}), since any group that contains an irreducible element must itself be irreducible.} Then it is natural to define a \defn{Singer cycle} in $G$ to be an irreducible element of maximum order.  The following questions arise naturally from our work.

\begin{question}
    Are Singer cycles strongly quasi-Coxeter in other finite reflection groups over $\FF_q$?
\end{question}

\begin{question}
    Under what conditions do a Singer cycle $c$ and a reflection $t$ generate the whole group $G$?
\end{question}

\subsection{The weak and strong quasi-Coxeter property}

Let $(G, T)$ be a generated group.  Above, following \cite{LW}, we established the terminology that an element $g \in G$ is \emph{strongly quasi-Coxeter} if for every shortest $T$-factorization $(t_1, \ldots, t_k)$ of $g$, one has $\langle t_1, \ldots, t_k\rangle = G$.  Correspondingly, we have that $g$ is \defn{weakly quasi-Coxeter} if there exists at least one such shortest $T$-factorization of $g$.  These properties arise in the study of real reflection groups (i.e., finite Coxeter groups) in multiple independent ways; for a detailed discussion of their history and significance, see \cite[\S3]{DLM2}.

When $G$ is a finite Coxeter group and $T$ is its set of reflections, \cite{BGRW} showed that every weakly quasi-Coxeter element is in fact strongly quasi-Coxeter.  And \cite{LW} showed that this extends to the complex reflection groups in the case of elements of reflection length $\rank{G}$. (Every weakly quasi-Coxeter element must have reflection length at least $\rank{G}$; in real groups, there are no elements of reflection length $> \rank{G}$, while in some complex groups there exist weakly-but-not-strongly quasi-Coxeter elements of reflection length $> \rank{G}$.)  In $\GLnFq$, the corresponding statement is false: for example, in $\GL_2(\FF_5)$, the element $\bmat{3 & 0 \\ 0 & 4}$ is not strongly quasi-Coxeter (it has an obvious factorization as a product of two diagonal reflections that together generate an abelian group of order $8$) but it is weakly quasi-Coxeter: one has
\[
\bmat{3 & 0 \\ 0 & 4} = \bmat{2 & 2 \\ 2 & 0} \cdot \bmat{0 & 2 \\ 4 & 3},
\]
the two factors are reflections (the first fixing the span of $\bmat{1\\2}$ pointwise, the second doing the same for $\bmat{1\\3}$), and it is a straightforward computer check that they generate $\GL_2(\FF_5)$.

\subsection{A correction}
\label{sec:gill}

In \cite[Thm.~2]{Gill}, Theorem~\ref{thm:gill} is claimed without the exceptional case for $n = 2$.  However, the statement is not true in this case.  Indeed, one can reverse the argument in the proof of Theorem~\ref{conj:cycle and reflection} to produce an exceptional companion matrix from an exceptional reflection, as illustrated by the following example: 
in $G = \GL_2(\FF_3)$, let \[
c = \begin{bmatrix}
0 & 1 \\
1 & -1
\end{bmatrix}
\]
be the companion matrix of the primitive polynomial $x^2 + x - 1$, hence a Singer cycle, and let 
\[
t = \begin{bmatrix}
1 & 0 \\
-1 & -1
\end{bmatrix}
\]
be a reflection that interchanges the two eigenspaces of $c$, so that
\[
t^{-1} c t = \begin{bmatrix}
-1 & -1 \\
-1 & 0
\end{bmatrix}
= c^3.
\]
One has that 
\[
t' := c^5 t c^{-5} = \bmat{
1 & -1 \\
0 & -1}
\]
fixes the subspace $y = 0$, hence
\[
c \cdot t' = \begin{bmatrix}
0 & -1 \\
1 & 0
\end{bmatrix}
\]
is a companion matrix (of the degree-$2$ polynomial $x^2 + 1$, with nonzero constant term).  Then one can check that the subgroup
\[
S = \langle c, c\cdot t' \rangle
\]
is the normalizer of $\langle c \rangle$, with order $16 < 48 = |\GL_2(\FF_3)|$.
In a private communication, Gill has identified the error in the proof of \cite[Thm.~2]{Gill}; we describe it (and the correction) now.

In the application of \cite[Lem.~4]{Gill}, which asserts that a \emph{field extension subgroup} of $\GL_{ad}(\Fq)$ of degree $d$ and dimension $a$ contains no nontrivial elements with fixed space dimension $> a(d - 1)$, to prove \cite[Cor.~5]{Gill}, which asserts that no field extension subgroup of $\GLnFq$ ($n \geq 2$) contains two companion matrices of monic polynomials of degree $n$, the implication 
\[
ad = n \quad \Longrightarrow a(d - 1) < n - 1
\]
is implicitly invoked for positive integers $a, d, n$.  However, this inequality fails when $a = 1$.  In the context of \cite[Lem.~4]{Gill}, this precisely corresponds to omitting consideration of the possibility that when $c = C_f$ is a Singer cycle, the normalizer $N$ of $S := \langle c \rangle$ in $\GLnFq$ contains the companion matrix of some other polynomial.  In order to establish a correct version of \cite[Cor.~5]{Gill}, and consequently Theorem~\ref{thm:gill}, it remains to show that this (another companion matrix in $N$) can only occur when $n = 2$.  So suppose we have such a second companion matrix in $y = C_g \in N$.  Observe that $\dim\fix(cy^{-1}) = n-1$ (since the first $n - 1$ columns of $c$ and $y$ are the same).

We consider two cases.  If $y$ is also in $S$, then $cy^{-1}$ is a power of $c$, with fixed space dimension $n - 1 > 0$.  But the only power of $c$ with eigenvalue $1$ is the identity matrix, so $y = c$, a contradiction.  Alternatively, if $y$ is not in $S$, then the proof of \cite[Lem.~4]{Gill} shows that, in fact, $\dim\fix(cy^{-1})$ is at most $m$ where $m$ is a proper divisor of $n$. Thus $n-1$ is a proper divisor of $n$ and we have $n=2$.  The analysis of the case $n = 2$ is a trivial modification of Proposition~\ref{prop:normalizing reflections when n = 2}.

Using the corrected version of \cite[Cor.~5]{Gill} leads to the corrected version of \cite[Thm.~2]{Gill} given as Theorem~\ref{thm:gill} above.

\section*{Acknowledgements}

We thank Nick Gill for the discussions described in Section~\ref{sec:gill}, and for allowing us to include them here.  This research was supported in part by the Simons Foundation Travel Support for Mathematicians program.

\bibliography{GLnFq}{}

\begin{thebibliography}{BGRW17}

\bibitem[Ber00]{Bereczky}
\'{A}ron Bereczky.
\newblock Maximal overgroups of {S}inger elements in classical groups.
\newblock {\em J. Algebra}, 234(1):187--206, 2000.

\bibitem[BGRW17]{BGRW}
Barbara Baumeister, Thomas Gobet, Kieran Roberts, and Patrick Wegener.
\newblock On the {H}urwitz action in finite {C}oxeter groups.
\newblock {\em J. Group Theory}, 20(1):103--131, 2017.

\bibitem[Bro14]{Brookfield}
Tom Brookfield.
\newblock {\em Overgroups of a linear {S}inger cycle in classical groups}.
\newblock PhD thesis, University of Birmingham, 2014.

\bibitem[Die55]{Dieudonne}
Jean Dieudonn\'{e}.
\newblock Sur les g\'{e}n\'{e}rateurs des groupes classiques.
\newblock {\em Summa Brasil. Math.}, 3:149--179, 1955.

\bibitem[DLM24]{DLM2}
Theo Douvropoulos, Joel~Brewster Lewis, and Alejandro~H. Morales.
\newblock Hurwitz numbers for reflection groups {II}: {P}arabolic
  quasi-{C}oxeter elements.
\newblock {\em J. Algebra}, 641:648--715, 2024.

\bibitem[Gil16]{Gill}
Nick Gill.
\newblock On a conjecture of {D}egos.
\newblock {\em Cah. Topol. G\'{e}om. Diff\'{e}r. Cat\'{e}g.}, 57(3):229--237,
  2016.

\bibitem[HLR17]{HLR}
Jia Huang, Joel~Brewster Lewis, and Victor Reiner.
\newblock Absolute order in general linear groups.
\newblock {\em J. Lond. Math. Soc. (2)}, 95(1):223--247, 2017.

\bibitem[Hou18]{Hou}
Xiang-dong Hou.
\newblock {\em Lectures on finite fields}, volume 190 of {\em Graduate Studies
  in Mathematics}.
\newblock American Mathematical Society, Providence, RI, 2018.

\bibitem[Hup70]{Huppert}
Bertram Huppert.
\newblock Singer-{Z}yklen in klassischen {G}ruppen.
\newblock {\em Math. Z.}, 117:141--150, 1970.

\bibitem[KM97]{KemperMalle}
G.~Kemper and G.~Malle.
\newblock The finite irreducible linear groups with polynomial ring of
  invariants.
\newblock {\em Transform. Groups}, 2(1):57--89, 1997.

\bibitem[Led49]{ledermann}
Walter Ledermann.
\newblock {\em Introduction to the {T}heory of {F}inite {G}roups}.
\newblock Oliver and Boyd, Edinburgh-London; Interscience Publishers, Inc., New
  York, 1949.

\bibitem[LM16]{LewisMorales}
J.~B. Lewis and A.~Morales.
\newblock {${\rm GL}_n(\bold{F}_q)$}-analogues of factorization problems in the
  symmetric group.
\newblock {\em European J. Combin.}, 58:75--95, 2016.

\bibitem[LRS14]{LRS}
J.~B. Lewis, V.~Reiner, and D.~Stanton.
\newblock Reflection factorizations of {S}inger cycles.
\newblock {\em J. Algebraic Combin.}, 40(3):663--691, 2014.

\bibitem[LV24]{LVinroot}
Joel~Brewster Lewis and C.~Ryan Vinroot.
\newblock Counting reflection factorizations of {S}inger cycles in finite
  linear and unitary groups, 2024+.
\newblock In preparation.

\bibitem[LW22]{LW}
Joel~Brewster Lewis and Jiayuan Wang.
\newblock The {H}urwitz action in complex reflection groups.
\newblock {\em Comb. Theory}, 2(1):Paper No. 12, 34, 2022.

\bibitem[RSW04]{CSP}
V.~Reiner, D.~Stanton, and D.~White.
\newblock The cyclic sieving phenomenon.
\newblock {\em J. Combin. Theory Ser. A}, 108(1):17--50, 2004.

\bibitem[RSW06]{ReinerStantonWebb}
V.~Reiner, D.~Stanton, and P.~Webb.
\newblock Springer's regular elements over arbitrary fields.
\newblock {\em Math. Proc. Cambridge Philos. Soc.}, 141(2):209--229, 2006.

\end{thebibliography}
\bibliographystyle{alpha}

\end{document}